\documentclass[12pt]{amsart}
\usepackage{fullpage}
\usepackage{latexsym}
\usepackage{amssymb}
\usepackage[all]{xy}

\newtheorem{thm}{Theorem}[section]
\newtheorem{prop}[thm]{Proposition}
\newtheorem{cor}[thm]{Corollary}

\newtheorem{lem}[thm]{Lemma}

\newtheorem{prob}[thm]{Problem}

\newtheorem{exam}[thm]{Example}

\newtheorem{rem}[thm]{Remark}

\usepackage{hyperref}

\newcommand{\N}{\mathbb{N}}
\newcommand{\Q}{\mathbb{Q}}

\newcommand{\Z}{\mathbb{Z}}

\newcommand{\SO}{\mathrm{\SO}}

\title[Realising sets of integers as mapping degree sets]
{Realising sets of integers as mapping degree sets}

\author{Christoforos Neofytidis}
\address{Department of Mathematics, Ohio State University, Columbus, OH 43210, USA}
\email{neofytidis.1@osu.edu}
\author{Shicheng Wang}
\address{School of Mathematical Sciences, Peking University, Beijing 100871, China}
\email{wangsc@math.pku.edu.cn}
\author{Zhongzi Wang}
\address{School of Mathematical Sciences, Peking University, Beijing 100871, China}
\email{wangzz22@stu.pku.edu.cn}

\date{\today}

\begin{document}

\maketitle

\begin{abstract} Given two closed oriented manifolds $M,N$ of the same dimension, we denote the set of degrees of maps from $M$ to $N$ by $D(M,N)$. The set $D(M,N)$ always contains zero.
We show the following (non-)realisability results:
\begin{itemize}
\item[(i)] There exists an infinite subset $A$ of $\Z$ containing $0$ which cannot be realised as $D(M,N)$, for any closed oriented $n$-manifolds $M,N$. 
\item[(ii)] Every finite arithmetic progression
 of integers containing $0$ can  be realised as $D(M,N)$, for some closed oriented $3$-manifolds $M,N$.
 \item[(iii)] Together with $0$, every finite  geometric progression of positive integers starting from $1$ can  be realised as $D(M,N)$, for some closed oriented manifolds $M,N$.
\end{itemize}
\end{abstract}

\tableofcontents
\section{Introduction} 

Let $M,N$ be two closed oriented manifolds  of the same dimension.
The mapping degree of a map $f\colon M\to N$, denoted by $\deg(f)$, is 
probably one of the  oldest and most fundamental concepts in topology.
The set of degrees of maps from $M$ to $N$, defined by
\[
D(M,N):=\{d\in\Z \ | \ \exists \ f\colon M\to N, \ \deg(f)=d\}.
\]
builds a bridge from topology to number theory: Each ordered pair  of manifolds $M, N$  as above gives a subset $D(M,N)$ of the integers.

Calculating or estimating $D(M,N)$  for various classes of manifolds $(M,N)$ is a topic with a long history and applications, and it is still very active to date. Some fairly recent examples include computations for infinite self-mapping degree sets of 3-manifolds~\cite{SWWZ}, computations and estimates for self-mapping degrees for products together with connections to the individual self-mapping degrees of their factors~\cite{Ne1}, as well as for simply connected targets, such as the conjectured unboundedness of some $D(M,N)$ for each simply connected manifold $N$~\cite{CMV}. For a much richer discussion and results, we refer the reader to the references in the aforementioned papers.

Conversely, the problem of realising arbitrary sets of integers as mapping degrees does not seem to have been rigorously addressed thus far. More precisely, the following question is widely open:

\begin{prob}\label{problem} 
Given a set $A\subseteq\Z$ with $0\in A$, are there closed oriented manifolds $M$ and $N$ such that $D(M,N)=A$?
\end{prob}

\begin{rem}
Note that the condition $0\in A$ is clearly necessary, because the constant map $M\to N$ realises $0\in D(M,N)$ for any $M,N$. Another, more restrictive question related to Problem \ref{problem} is about self-mapping degrees: Given a set $A\subseteq\Z$ with $0,1\in A$ and $ab\in A$ whenever $a,b\in A$, is there a closed oriented manifold $M$ such that $D(M,M)=A$?  Again, the additional requirements $1\in A$ and $ab\in A$ whenever $a,b \in A$, are clearly necessary, because $1\in D(M,M)$ is realised by the identity map, and $ab\in D(M,M)$ is realised by composing two self-maps of $M$ of degrees $a$ and $b$.

\end{rem}

Problem \ref{problem} has been circulated for years; among other, the first two authors have been asked or have asked this question several times while delivering public lectures on the topic of mapping degree.
However, no answer had been given. 
In our first result, we answer Problem \ref{problem} in the negative.

\begin{thm}\label{non} 
There exists an infinite subset $A\subseteq\Z$ containing zero which cannot be realized as $D(M,N)$, for any closed oriented $n$-manifolds $M,N$. 
\end{thm}

Our result  is in fact stronger,  contrasting the amount of arbitrary sets of integers with those that arise from purely topological data (i.e. homotopy types and mapping degrees), showing thus that ``most" arbitrary infinite subsets of $\Z$ (containing zero) are not realizable as mapping degree sets. 
Thus, we suggest a refined version of Problem \ref{problem}:

\begin{prob}\label{finite}
Suppose $A$ is a finite set of integers containing zero. Does $A=D(M,N)$ for some closed $n$-manifolds $M$ and $N$?
\end{prob}

To obtain some better intuition for $D(M,N)$, we review several simple cases  in the following example. 
For a finite set $A$, we use $|A|$ to denote the cardinality of $A$.

\begin{exam} \label{simple}
Suppose $M$ and $N$ are closed oriented $n$-manifolds.
\begin{itemize}
\item[(i)] If $n=1$, then $D(M,N)=\Z$.
\item[(ii)] If $n=2$, then $D(M,N)$ is either $\Z$ or the integer interval $[-k, k]$ for some $k\ge 0$.
\item[(iii)] If $N$ is covered by the $n$-sphere $S^n$, then 
\[
D(M,N)=\{d+|\pi_1(N)|\Z \ |\ \text{for some integers} \ d \in [1 ,  |\pi_1(N)|]\}.
\]
\end{itemize}
\end{exam}

The above results are known. We give an argument for the less well-known case (iii)
(see also \cite{Ol} or \cite[Theorem 1]{SWWZ}): The degree of the covering $S^n\to N$ is $|\pi_1(N)|$. Since $D(S^n,S^n)=\Z$, we obtain  $|\pi_1(N)|\Z\subseteq D(S^n,N)$. If $l\in D(M,N)$, then $l+|\pi_1(N)|\Z\subseteq D(M,N)$, because $M=M\#S^n$ (see Lemma \ref{l:degreeconnected}). Thus, $D(M,N)=\{l+|\pi_1(N)|\Z \ | l\in D(M,N)\}$. Since for each $l\in \Z$, $\{l+|\pi_1(N)|\Z\}=\{d  +|\pi_1(N)|\Z\}$ for some $d \in [1 ,  |\pi_1(N)|]$, case (iii) follows. 

Cases (i) and (ii) in Example \ref{simple} are arithmetic progressions (infinite or finite) of constant difference 1, and case (iii) is a union of finitely many infinite arithmetic progressions  of constant difference $|\pi_1(N)|$.  
These observations motivate the following question -- also a refinement of Problem \ref{problem} -- from a number theoretic point of view:
 
\begin{prob}\label{arith} 
Can every  arithmetic progression  containing zero  be realised as $D(M,N)$ for some closed oriented $n$-manifolds $M,N$?
\end{prob}

We give an affirmative answer to Problem \ref{arith} for finite sets:

\begin{thm}\label{main1} 
Every finite arithmetic progression of integers containing zero can be realised as $D(M,N)$ for some closed oriented $3$-manifolds $M,N$.
\end{thm}

Theorem \ref{main1} will be a corollary of the more general realisation Theorem \ref{main}, which is probably somehow involved to be stated in the introduction. As we shall see in Section \ref{s:proofrealizability}, Theorem \ref{main} has also other consequences concerning  Problem \ref{finite}.

Prompted by Problem \ref{arith} and Theorem \ref{main1}, we further ask the following:

\begin{prob}\label{geom}
Together with 0,  can every geometric progression of integers  be realised as $D(M,N)$ for some closed oriented $n$-manifolds $M,N$?
\end{prob}

We  give a slightly more restrictive (compared to the case of arithmetic progressions), but still substantial, answer to Problem \ref{geom}:

\begin{thm}\label{main3} 
Together with 0,  every finite geometric progression of positive integers starting from 1 can be realised as $D(M,N)$ for some closed oriented manifolds $M,N$.
\end{thm}

Theorem \ref{main3} will also follow from a more general realisation result (Theorem \ref{productdegrees}).

\vskip 0.2 true cm

\noindent{\bf Ideas of the proofs:} The ideas for the proofs of the above results can be outlined quickly: (i) The proof of Theorem \ref {non} is based on the idea of using countability; (ii) Both 3-manifolds $M$ and $N$ in Theorem \ref{main} (Theorem \ref{main1})  will be connected sums of certain circle bundles over surfaces with non-zero Euler classes, which in turn determine the mapping degree sets between those circle bundles (Lemma \ref{l:degreecircle}); (iii) Both manifolds $M$ and $N$ in Theorem \ref{productdegrees} (Theorem \ref{main3})  will be products of 3-manifolds which are of the forms stated in (ii). 
As we shall see in the course of the proofs, both the constructions and verifications in (ii) and (iii) are somewhat delicate, especially for Theorem \ref{productdegrees}.

\begin{rem}
Since in all of the constructions in this paper we will be using aspherical $3$-manifolds as building blocks,  our manifolds will have non-trivial fundamental groups. Thus, a further natural refinement of Problem \ref{problem} and of its variations would be to consider similar realisability questions for simply connected manifolds. 
\end{rem}

\noindent{\bf Acknowledgments.}
When parts of this project were carried out in the summer of 2021,
C. Neofytidis was visiting  MPIM Bonn and the University of Geneva, S.C. Wang and Z.Z. Wang were visiting IASM of Zhejiang University.
We thank all these institutes.

Professors Yi Liu, Shengkui Ye, and especially  Jianzhong Pan, helped us to provide  a proof on a first draft of this paper that there are only countably many homotopy types of closed manifolds. Subsequently, Professors Jean-Fran\c cois Lafont and Shmuel Weinberger pointed out that this latter fact is a theorem of M. Mather \cite{Ma}.
We thank all of them.

Finally, we thank the anonymous referees for their suggestions.

\section{Non-realisability for infinite sets}\label{s:non-realisable}

Theorem \ref{non},  a negative answer to Problem \ref{problem},  now follows quickly from the idea of using countability.
In fact, if we restrict to closed oriented smooth manifolds, the proof becomes very elementary.

\begin{proof}[Proof of Theorem  \ref{non}] 
Let $\Z^*$ be the set of all non-zero integers.  Since  $\Z^*$ has uncountably many subsets and countably many
finite sets, it has uncountably many infinite subsets. In particular, $\Z$ has uncountably many infinite subsets containing zero.
Thus, in order to prove Theorem \ref{non}, we only need  to prove the following:

\medskip

\noindent{\bf Claim:} For every $n$, there are only countably many integer 
sets $D(M,N)$ of pairs of closed oriented $n$-manifolds $(M,N)$.

\medskip

We first prove the Claim  for triangulable closed oriented $n$-manifolds, which is elementary, and 
already contains all closed oriented smooth or piecewise linear manifolds.

First, fix the dimension $n$. For each integer  $k\ge 0$, there are only finitely many simplical complexes consisting of $k$ simplices.
In particular, there are only finitely many closed $n$-manifolds consisting of $k$ simplices. By induction on $k$, there are only countably many closed triangulable $n$-manifolds. Thus, there are only countably many pairs $(M,N)$ of closed triangulable $n$-manifolds.
Then, by induction on $n$, there are only countably many pairs $(M,N)$ of closed triangulable $n$-manifolds in all dimensions $n$.
It follows that there are only countably many integer sets $D(M,N)$ of closed oriented triangulable $n$-manifolds $(M,N)$ in all dimensions $n$.

Now we discuss the general case.
Let $M$, $N$, $X$ and $Y$ be closed oriented $n$-manifolds. Suppose $X$ and $Y$  are homotopy equivalent to 
$M$ and $N$ respectively. Then 
\[
D(M,N)=D(X,Y).
\]
Following the argument given in the triangulable case, we need to prove that there are only countably many homotopy classes of closed 
oriented $n$-manifolds. This is a theorem of Mather~\cite[Corollary, p. 93]{Ma}.
\end{proof}
 
\section{Realisability for finite arithmetic progressions}\label{s:proofrealizability}

 Theorem \ref{main1} is a special case of the following more general realisation result, which will be proven in the end of this section.
 
\begin{thm}\label{main}
For any $k\in \N_+$ and any integers 
\[
d_1,d_2,...,d_k>0 \ \text{and} \
n_1,n_1',n_2,n_2',...,n_k,n_k'\ge 0,
\] 
there exist closed oriented $3$-manifolds $M,N$ such that 
\[
D(M,N)=\{d\in\mathbb{Z} \ | \ d=\sum_{i=1}^{k}{m_id_i},\,\,-n_i'\le m_i \le n_i\}.
\]
\end{thm}

Corollary \ref{interval} below is  a general form of Theorem \ref{main1}.
A  finite sequence of integer intervals 
\[
\{[b_i, c_i], i=1,2,...,l\}
\]
is called arithmetic, if the lengths of all $[b_i, c_i]$ are equal,
and all the differences $b_{i+1}-b_i$ are equal. When $b_i=c_i$, we obtain a usual finite arithmetic progression.

\begin{cor}[Theorem \ref{main1}]\label{interval} 
Every finite arithmetic sequence of integer intervals containing zero can be realised as $D(M,N)$ for some closed $3$-manifolds $M,N$. In particular, every finite arithemetic progression containing zero is realisable as a mapping degree set.
\end{cor}

\begin{proof}
Suppose 
$\{[b_i, c_i], i=1,2,...,l\}$ is a finite arithmetic sequence of integer  intervals,
where $b_i\le c_i< b_{i+1}$, and $0\in [b_k, c_k]$ for some $1\le k\le l$.
Let
\[
n_1=c_k,\,\,  n_1'=-b_k,\,\, d_2=b_2-b_1,\,\, n_2=l-k,\,\, n_2'=k-1.
\]
 Since $\{b_i, \ i=1,...,l\}$ is an arithmetic sequence with constant difference $d_2$, we have 
$b_i=b_k+d_2(i-k)=-n_1'+d_2(i-k)$.
Similarly, $c_i=c_k+d_2(i-k)=n_1+d_2(i-k)$.
Thus,
\begin{equation*}
\begin{aligned}
A=\bigcup_{i=1}^{l}{[b_i,c_i]}
&=\bigcup_{i=1}^{l}[-n_1'+d_2(i-k),n_1+d_2(i-k)]\\
&=\bigcup_{j=1-k}^{l-k}[-n_1'+d_2j,n_1+d_2j]\\
&=\bigcup_{j=-n_2'}^{n_2}[-n_1'+d_2j,n_1+d_2j]\\
&=\bigcup_{j=-n_2'}^{n_2}\{d\in\Z \ | \ d=m_1+jd_2,\,\,-n_1'\le m_1\le n_1\}\\
&=\bigcup_{m_2=-n_2'}^{n_2}\{d\in\Z \ | \ d=m_1+m_2d_2,\,\,-n_1'\le m_1\le n_1\}\\
&=\{d\in\Z \ | \ d=m_1+m_2d_2,\,\,-n_i'\le m_i\le n_i\}.
\end{aligned}
\end{equation*}
The proof follows by Theorem \ref{main} for $k=2$ and $d_1=1$.
\end{proof}

Another consequence of Theorem \ref{main} is the following:

\begin{cor}\label{main2} 
 Let $A=\{d_1,...,d_l\}$ be a finite set of integers containing zero. There are closed oriented $3$-manifolds $M$ and $N$ such that 
\begin{equation*}\label{eq.A=D}
D(M,N)=\biggl\{\sum_{j\in S}d_j \ | \ S\subseteq\{1,...l\}\biggl\}.
\end{equation*}
\end{cor}
\begin{proof}\
Set
$
n_1'=\cdots=n_k'=0,\,\, n_1=\cdots=n_k=1
$
in Theorem \ref{main}.
\end{proof}

Now we are going to prove Theorem \ref{main}. We need some more preparations.

\medskip

Given a circle bundle $S^1\to K\to\Sigma$, where $\Sigma$ is a closed oriented surface, the Euler number of $K$ is defined by the Kronecker product
\[
\hat{e}(K)=\langle e(K),[\Sigma]\rangle,
\]
where $e(K)\in H^2(\Sigma;\Z)=\Z$ denotes the Euler class of $K$.

The following lemma determines the mapping degree sets when running over all Euler numbers for a fixed hyperbolic surface. 

\begin{lem}\label{l:degreecircle}
Let $\Sigma$ be a closed oriented hyperbolic surface and $K_i\stackrel{p_i}\longrightarrow\Sigma$ be the circle bundle with Euler number $\hat{e}(K_i)=i$. Then
\begin{equation}\label{eq.mappingcircle}
D(K_i,K_j)= \left\{\begin{array}{ll}
       \{0,\frac{j}{i}\}, & \text{if} \ i\mid j\\
       \text{} & \\
        \{0\}, & \text{if} \ i\nmid j.
        \end{array}\right.
\end{equation}
Moreover, all of the non-zero degree maps are homotopic to coverings.
\end{lem}
\begin{proof} For $s=i,j$, we have a surjection $\pi_1(K_s)\stackrel{{p_s}_*}\longrightarrow\pi_1(\Sigma)$
with kernel  $\Z=[t]$ represented by an $S^1$ fiber $t$ (see \cite[Lemma 3.2]{Sc}), so that this normal subgroup $\Z\subseteq\pi_1(K_s)$ belongs to the center $Z(\pi_1(K_s))$ of $\pi_1(K_s)$; see~\cite[p. 118]{He}. Now, if $x\in  Z(\pi_1(K_s))$, then ${p_s}_*(x)\in Z(\pi_1(\Sigma))$.
 Since $\Sigma$ is a hyperbolic surface, $Z(\pi_1(\Sigma))$ is trivial, therefore $x$ is in the kernel of ${p_s}_*$, that is, $x\in \Z$.
 Thus, $Z(\pi_1(K_s))=\pi_1(S^1)=\Z$. Note that this fact can be also obtained from \cite[Sec. IV. 3]{Br}.

Let $f\colon K_i\to K_j$ be a map of non-zero degree. Since the center of $\pi_1(\Sigma)$ is trivial, after lifting $f$ to a $\pi_1$-surjective map $K_i\to\overline{K_j}$ (where $\overline{K_j}$ is the cover of $K_j$ corresponding to $f_*(\pi_1(K_i))$), we deduce that the center of $\pi_1(K_i)$ is mapped trivially in $\pi_1(\Sigma)$ under the induced homomorphism $(p_2\circ f)_*\colon\pi_1(K_i)\to\pi_1(\Sigma)$. Thus, by the asphericity of our spaces, there is a map $\bar{f}\colon\Sigma\to\Sigma$ such that $\bar{f}\circ p_1=p_2\circ f$ up to homotopy.

Since $\deg(f)\neq0$, we conclude that $\deg(\bar{f})\neq0$. Hyperbolic surfaces do not admit self-maps of degree greater than one, hence $\deg(\bar{f})=\pm1$. In particular $\bar{f}$ is $\pi_1$-surjective. Since $\pi_1(\Sigma)$ is Hopfian, we conclude that $\bar{f}$ induces an isomorphism on $\pi_1(\Sigma)$ and thus, since $\Sigma$ is aspherical, $\bar{f}$ is a homotopy equivalence. The Borel conjecture is true for aspherical surfaces, hence $\bar{f}$ is homotopic to a homeomorphism. Since every self-map of the circle is homotopic to a covering map, we deduce that $f$ is homotopic to a fiber-preserving covering of degree 
\[
\deg(f)=\deg(\bar{f})\deg(f|_{S^1})=\pm\deg(f|_{S^1}).
\]
Moreover, by~\cite{NR} (see~\cite[Theorem 3.6]{Sc}), we obtain
\[
\hat{e}(K_i)=\hat{e}(K_j)\frac{\deg(\bar{f})}{\deg(f|_{S^1})}=\frac{\hat{e}(K_j)}{\deg(f)}.
\]
This can happen only if $i\mid j$. We deduce that
\[
D(K_i,K_j)\subseteq\biggl\{0,\frac{j}{i}\biggl\}, \ \text{if} \ i\mid j, \ \text{and} \ D(K_i,K_j)=\{0\}, \ \text{if} \ i\nmid j.
\]
We still need to show that $\frac{j}{i}\in D(K_i,K_j)$, whenever $\frac{j}{i}\in\Z$ (see~\cite[Example 1.4]{Ne2}): Since $K_i$ is fiberwise oriented, it is a principal $U(1)$-bundle, and hence can be viewed as the associated complex line bundle whose first Chern number is $c_1(K_i)=\hat{e}(K_i)=i$.  The tensor product of $\frac{j}{i}$ copies of $K_i$ has first Chern number
\[
c_1(\otimes^{\frac{j}{i}}K_i)=\frac{j}{i}c_1(K_i)=\frac{j}{i}\hat{e}(K_i)=j=\hat{e}(K_j).
\]
Hence, $\otimes^{\frac{j}{i}}K_i\cong K_j$. The $\frac{j}{i}$-th power of a section of $K_i$ gives us a fiberwise covering map $$f\colon K_i\to\otimes^{\frac{j}{i}}K_i,$$ which is of degree $\frac{j}{i}$ on the $S^1$-fibers and of degree one on $\Sigma$. In particular, 
\[
\deg(f)=\frac{j}{i}\in D(K_i,K_j),
\]
showing (\ref{eq.mappingcircle}).
\end{proof}

Recall that given sets of integers $A_i$, $i=1,...,k$, the sum of $A_i$ is defined to be
\[
\sum_{i=1}^k A_i =\biggl\{\sum _{i=1}^k a_i \,| \,a_i\in A_i\biggl\}.
\]
When $A_1,...,A_k$ are equal to the same $A$, we often denote $\sum_{i=1}^k A_i$ by $\sum^k A$.

\medskip

The next lemma provides a connection between $D(M_1\#M_2,N)$ and $D(M_1,N)+D(M_2,N)$.

\begin{lem}\label{l:degreeconnected}
Let $M_1,M_2$ and $N$ be closed oriented manifolds of dimension $n$. Then
\begin{equation}\label{eq.connected}
D(M_1,N)+D(M_2,N)\subseteq D(M_1\#M_2,N),
\end{equation}
with equality if $\pi_{n-1}(N)=0$. 
\end{lem}
\begin{proof}
For $i=1,2$, let $f_i\colon M_i\to N$ be maps of degree $d_i$. Consider the following composite map
\[
f\colon M_1\#M_2 \stackrel{q}\longrightarrow M_1\vee M_2\stackrel{f_1\vee f_2}\longrightarrow N\vee N\stackrel{h}\longrightarrow N,
\]
where $q$ is the map that pinches the connecting $S^{n-1}$ to a point and $h$ is a homeomorphism that maps each copy of $N$ to itself. Then in degree $n$ homology
\begin{equation*}
\begin{aligned}
H_n(f)([M_1\#M_2])  & =H_n(h)\circ H_n(f_1\vee f_2)\circ H_n(q)([M_1\#M_2])\\
                                         & = H_n(h)\circ H_n(f_1\vee f_2)([M_1],[M_2])\\
                                         & = H_n(h)(d_1[M_1],d_2[M_2])\\
                                         & =(d_1+d_2)[N],
\end{aligned}
 \end{equation*}
which shows inclusion (\ref{eq.connected}).

Suppose now $\pi_{n-1}(N)=0$ and let $f\colon M_1\#M_2\to N$ be a map of non-zero degree. Since any map $S^{n-1}\to N$ is null-homotopic, we deduce that $f$ factors through the pinch map $q\colon M_1\#M_2\to M_1\vee M_2$, that is, there is a continuous map $g\colon M_1\vee M_2\to N$ such that $f=g\circ q$. Hence, in degree $n$ homology we have
\begin{equation*}
\begin{aligned}
\deg(f)[N]  & =H_n(f)([M_1\#M_2])\\
                                         & =H_n(g)\circ H_n(q)([M_1\#M_2])\\
                                         & =H_n(g)([M_1],[M_2])\\
                                         & =(d_1+d_2)[N],
\end{aligned}
 \end{equation*}
where $H_n(g|_{M_i})([M_i])=d_i[N]$, i.e. $d_i\in D(M_i,N)$, for $i=1,2$. This shows the inclusion
\[
D(M_1\#M_2,N)\subseteq D(M_1,N)+D(M_2,N).
\]
\end{proof}

We are now ready to prove Theorem \ref{main}:

\begin{proof}[Proof of Theorem \ref{main}]
Set 
\[
d'=d_1d_2...d_k,  \, \text{and} \, d_i'=d'/d_i, \ i=1,...,k.
\]
Let
\[
N=K_{d'}, \ M_i=K_{d_i'} \ \text{and} \ M_i'=K_{-d_i'}
\]
be circle bundles over a closed oriented hyperbolic surface $\Sigma$ with Euler numbers
\[
\hat{e}(N)=d', \ \hat{e}(M_i)=d_i' \ \text{and} \ \hat{e}(M_i')=-d_i'
\]
respectively.

Since $d'/d'_i=d_i$, Lemma \ref{l:degreecircle} tells us that
\begin{equation}\label{eq.d_i}
D(M_i,N)=D(K_{d'_i}, K_{d'})=\{0,d_i\}.
\end{equation}
Similarly,
\begin{equation}\label{eq.-d_i}
D(M_i',N)=\{-d_i,0\}.
\end{equation}

Let 
\begin{equation*}
M=\#_{i=1}^{k}((\#_{n_i}{M_i})\#(\#_{n_i'}{M_i'})).
\end{equation*}
Since $N$ is aspherical, in particular $\pi_2(N)=0$, we apply  Lemma \ref{l:degreeconnected} successively to obtain

\[
D(M,N)=\sum_{i=1}^{k}{(\sum_{j_i=1}^{n_i}{D(M_i,N)}+\sum_{j_i=1}^{n_i'}D(M_i',N))}.
\]
By (\ref{eq.d_i}) and (\ref{eq.-d_i}),
$\sum_{j_i=1}^{n_i}{D(M_i,N)}+\sum_{j_i=1}^{n_i'}{D(M_i',N)}$ is the sum of  $n_i$ copies of $\{0, d_i\}$ and of $n'_i$ copies of $\{0, -d_i\}$.
Hence,
\[
\sum_{j_i=1}^{n_i}{D(M_i,N)}+\sum_{j_i=1}^{n_i'}{D(M_i',N)}=\{m_id_i \ |  -n_i'\le m_i\le n_i\}.
\]
We conclude that
$$D(M,N)=\{d\in\Z \ | \ d=\sum_{i=1}^{k}{m_id_i},\,\,-n_i'\le m_i\le n_i\},$$
finishing the proof of Theorem \ref{main}.
\end{proof}

\section{Realisability for finite geometric progressions}\label{s:proofrealizability2}

 Theorem \ref{main3} about finite geometric progressions
  is a straightforward consequence of the following more general realisability result.

\begin{thm}\label{productdegrees}
Given integers $1\leq d_1 \leq d_2 \leq\cdots \leq d_l$, there exist closed oriented $3l$-manifolds $M$ and $N$ such that 
\[
D(M,N)=\{0,1\}\cup\biggl\{\prod_{j\in S}d_j \ | \ \emptyset\ne S\subseteq \{1,2,...,l\}\biggl\}.
\]
\end{thm}

\begin{proof}[Proof of Theorem \ref{main3} from Theorem \ref{productdegrees}]
Let $d_1=d_2=\cdots=d_l=d$.  Then Theorem \ref{productdegrees} implies
 \[
 D(M, N)=\{0, 1, d,  d^2,..., d^l\}.
 \]
\end{proof}

We will devote the rest of this section to the proof of Theorem \ref{productdegrees}.

\medskip

For brevity, we say that a closed oriented $n$-manifold $M$ {\em dominates} (resp. 1-{\em dominates}) another closed oriented $n$-manifold $N$ if there exists a  map  $f\colon M\to N$ of  non-zero degree (resp. of degree one).

We begin with some easy observations:

\begin{lem}\label{connectedsumdomination}
Given any closed oriented $n$-manifolds $M$ and $N$, there is a $1$-domination $M\# N\to N$.
\end{lem}

\begin{proof} 
This follows from Lemma \ref{l:degreeconnected}; in fact it is contained in the proof of Lemma \ref{l:degreeconnected}. Namely, consider the following composite map
\[
M\#N \stackrel{q}\longrightarrow M\vee N\stackrel{h}\longrightarrow N,
\]
where $q$ pinches the connecting $S^{n-1}$ to a point, and $h$ sends $M$ to that point.
\end{proof}

We denote the degree 1 map $M\# N\to N$ in Lemma \ref{connectedsumdomination} by $p$ and also call it a {\em pinch map}.

\begin{lem}\label{connectedsumdegrees}
Let $M,N_1$ and $N_2$ be closed oriented $n$-manifolds. Then 
\[
D(M,N_1\#N_2)\subseteq D(M,N_1).
\]
\end{lem}

\begin{proof}
 Suppose $l\in D(M,N_1\#N_2)$ and $f\colon M\to N_1\#N_2$ be a map of degree $l$. Let the composition
 \[
M \stackrel{f}\longrightarrow N_1\# N_2 \stackrel{p}\longrightarrow N_1,
\]
where $p$ is the pinch map given by Lemma \ref{connectedsumdomination}. Then $p\circ f$ is of degree $l$, so $l\in D(M,N_1)$.
\end{proof}

The following result is a special case of Theorem \ref{productdegrees}, as well as a crucial step to prove Theorem \ref{productdegrees}.

\begin{thm}\label{1,k}
For any integer $d>1$, there exist closed oriented $3$-manifolds $Q$ and $P$ such that $D(Q,P)=\{0,1,d\}$. 
\end{thm}

\begin{proof} 
Let $q>d$ be a prime number, and consider the following manifolds, where, as in Section \ref{s:proofrealizability}, $K_i$ denotes the $S^1$-bundle over a fixed hyperbolic surface with Euler number $i$:
\[
Q=(\#_d K_q)\# K_d\#K_{d^2} \ \ \text{and} \ \ P= K_q\#K_{d^2}.
 \]
 Let $Q_1=(\#_dK_q)\#K_d$. By Lemma \ref{l:degreecircle}, $K_d$ is a $d$-fold covering of $K_{d^2}$, and so we obtain a covering 
 \begin{equation}\label{eq.d-fold}
Q_1=(\#_dK_q)\#K_d\to K_q\#K_{d^2}=P
\end{equation}
of degree $d$. 
Note that
\[
Q=P\#(\#_{d-1}K_q)\#K_{d}=
Q_1\#K_{d^2}.
\]
By Lemma \ref{connectedsumdomination}, there are $1$-dominations $Q\to Q_1$ and $Q\to P$. Together with (\ref{eq.d-fold}), we deduce 
\begin{equation}\label{eq.inclusion1}
\{0,1,d\}\subseteq D(Q,P).
\end{equation}

We will now show the converse inclusion. Lemma \ref{connectedsumdegrees} implies that
\begin{equation}\label{eq.subsetintersection}
 D(Q,P)\subseteq D(Q,K_q)\cap D(Q,K_{d^2}).
 \end{equation}
Since $K_q$ is aspherical, in particular $\pi_2(K_q)=0$,  Lemma \ref{l:degreeconnected} implies that
 \[
 D(Q,K_q)=\sum^{d}D(K_q,K_q)+D(K_d,K_q)+D(K_{d^2},K_q).
 \]
Since $d$ and $q$ are coprime, Lemma \ref{l:degreecircle} tells us that
\[
D(K_q,K_q)=\{0,1\},
\]
\[D(K_d,K_q)=D(K_{d^2},K_q)=\{0\},
\] 
and thus 
\begin{equation}\label{eq.QKq}
D(Q, K_q)=\{0,1,...,d\}. 
\end{equation}
Applying the same argument we obtain
\begin{equation}\label{eq.QKd^2}
D(Q,K_{d^2})=\{0,1,d,d+1\}. 
\end{equation}
Then by (\ref{eq.subsetintersection}), (\ref{eq.QKq}) and (\ref{eq.QKd^2}) we have
\begin{equation}\label{eq.inclusion2}
D(Q,P)\subseteq\{0,1,...,d\}\cap\{0,1,d,d+1\}=\{0,1,d\}.
\end{equation}

The theorem follows by (\ref{eq.inclusion1}) and (\ref{eq.inclusion2}).
\end{proof}

Equipped with Theorem \ref{1,k}, we will be able to prove Theorem \ref{productdegrees} by using  products of suitable 3-manifolds. To do this we still 
need some preparations. 

\medskip

Recall that given sets of integers $A_i$, $i=1,...,k$, the product of $A_i$ is defined to be  

\[
\prod_{i=1}^k A_i =\biggl\{\prod_{i=1}^k a_i \,| \,a_i\in A_i\biggl\}.
\]
When $A_1,... ,A_k$ are equal to the same $A$, we often denote $\prod_{i=1}^k A_i$ by $\prod^k A$.

\medskip

We begin with a straightforward observation:

\begin{lem}\label{l:product} Given closed oriented $n$-manifolds $M,N$ and $m$-manifolds $W,Z$, we have
$$D(M,N)\cdot D(W,Z)\subseteq D(M\times W, N\times Z).$$
\end{lem}

\begin{proof}
Let $f\colon M\to N$ and $g\colon W\to Z$ be maps of degree $k$ and $l$ respectively.
By taking products of manifolds and maps, we obtain a map 
$
f\times g\colon M\times W\to N\times Z
$
of degree $kl$.
 \end{proof}

The converse inclusion to Lemma \ref{l:product} fails in general~\cite[Example 1.2]{Ne1}.
Nevertheless, Theorem \ref{t:product} below, which is a generalisation of~\cite[Theorem 1.4]{Ne1},  
gives some sufficient conditions so that equality holds. This
will be important in proving Theorem \ref{productdegrees}. 

\begin{thm}\label{t:product}
Let $M, N$ be two closed oriented manifolds of dimension $n$ and $W, Z$ of dimension $m$. Suppose
\begin{itemize}
\item[(i)] $N$ is not dominated by  direct products, and
\item[(ii)]   for any  map $W\to N$, the induced homomorphism $H_n(W, \Q)\to H_n(N; \Q)$ is trivial.
\end{itemize}
Then $D(M\times W, N\times Z)=D(M,N)\cdot D(W,Z)$.
\end{thm}

Before giving the proof of Theorem \ref{t:product}, we first make some remarks, mostly around Thom's work~\cite{Thom} on Steenrod's realisation problem.

\begin{rem} \
\begin{itemize}
\item[(1)] In \cite[Theorem 1.4]{Ne1}, condition (ii) is stated in cohomology, while 
in Theorem \ref{t:product} we chose to state condition (ii) in homology, since it is more direct in its application to the proof of Theorem \ref{productdegrees},
and also Thom's Realisation Theorem \cite{Thom}, which is needed in the proof Theorem \ref{t:product}, arises naturally in homology.
\item[(2)] Recall that Thom's Realisation Theorem states the following: Let $X$ be a topological space. For each  $\omega\in H_n(X; \Z)$, there is an integer $d>0$ and a map $f\colon M\to X$,
where $M$ is a closed oriented $n$-manifold, such that $H_n(f)([M])=d\omega$. In particular, each  $\omega\in H_n(X;\Q)$
can be realised by a closed oriented $n$-manifold.
\item[(3)] Even though Thom's Realization Theorem is crucial for the proof of Theorem \ref{t:product}, it will not be essential for the the proof 
of Theorem \ref{productdegrees}, since in Theorem \ref{productdegrees} one can see directly that each homology class can be realised by a closed oriented manifold.
\end{itemize}
\end{rem}

Next, we fix some notation for the proof of Theorem \ref{t:product}: $[M]$ and $[M]^*$ denote the integer fundamental classes of $H_n(M;\Q)$ and $H^n(M;\Q)$ respectively. Also, let $\iota_M\colon M\hookrightarrow M\times W$ be  the inclusion, $p_M\colon M\times W\longrightarrow M$  the projection, and denote $[M]\otimes 1= H_n({\iota_M})([M])$ and $\omega_M= H^n({p_M})([M]^*)$. Similar notation will be used for $W, N$ and $Z$.

\begin{proof}[Proof of Theorem \ref{t:product}]
By Lemma \ref{l:product}, it suffices to show the inclusion $D(M\times W, N\times Z)\subseteq D(M,N)\cdot D(W,Z)$. 
Let $f\colon M\times W\to N\times Z$ be a map of degree $d\neq0$. 
We have 
\[
H_l(f)\colon H_l(M\times W;\Q)\to H_l(N\times Z;\Q) \ \text{and} \ H^l(f)\colon H^l(N\times Z;\Q)\to H^l(M\times W;\Q)
\]
for $l\in \{0, 1, ... , m+n\}$. By the K\"unneth formula in homology,
we have 
\begin{equation}\begin{array}{rcll}\label{Kuenneth}
H_n(M\times W;\Q)=\oplus_{i=0}^{n}(H_{n-i}(M;\Q)\otimes H_{i}(W; \Q))=\Q<[M]\otimes 1>\oplus  V_{M}, \\ 
H_n(N\times Z;\Q)=\oplus_{i=0}^{n}(H_{n-i}(N;\Q)\otimes H_{i}(Z; \Q))=\Q<[N]\otimes 1>\oplus  V_{N},
\end{array}\end{equation}
where $V_{M}=\oplus_{i=1}^{n}(H_{n-i}(M; \Q)\otimes H_{i}(W; \Q))$ and $V_{N}=\oplus_{i=1}^{n}(H_{n-i}(N;\Q)\otimes H_{i}(Z; \Q))$.
 
Consider the composition
\[
M\times W\stackrel{f}\longrightarrow N\times Z\stackrel{p_N}\longrightarrow N.
\]
The restriction of $H_n(p_N\circ f)$ to $\oplus_{i=1}^{n-1}(H_{n-i}(M;\Q)\otimes H_{i}(W; \Q))$ 
maps trivially to 
$H_n(N; \Q)$ by condition (i) and Thom's Realisation Theorem,   and the restriction to $H_n(W; \Q)$  maps trivially  to
$H_n(N;\Q)$ by condition (ii). Hence, we have that
$H_n(p_N\circ f)(V_{M})=0$, which implies that 
\begin{equation}\label{eq.fV_M}
H_n(f)(V_{M})\subseteq V_{ N}. 
\end{equation}
Suppose now 
\begin{equation}\label{eq.fM}
H_n(f)([M]\otimes 1)=\kappa\cdot [N]\otimes 1+\delta, 
\end{equation}
for some $\kappa\in\Z$ and $\delta\in V_{N}$. Then $\kappa\in D(M,N)$ and a map of degree $\kappa$ is given by
\[
M\stackrel{\iota_M}\hookrightarrow M\times W\stackrel{f}\longrightarrow N\times Z\stackrel{p_N}\longrightarrow N.
\]

We are going to verify that (\ref{eq.fV_M}) and (\ref{eq.fM}) imply that 
\begin{equation}\label{eq.PDf}
H^n(f)(\omega_N)=\kappa\cdot \omega_M.
\end{equation}
Since $p_M$ and $p_N$ are projections, we have 
\[
H_n(p_M)(V_{M})=0\ \text{and}  \ H_n(p_N)(V_{N})=0.
\]
Thus,
\begin{equation}\label{eq.V_Mzero}
\langle\omega_M,V_{M}\rangle=\langle H^n(p_M)([M]^*),V_{M}\rangle=\langle[M]^*,H_n(p_M)(V_{M})\rangle=0
\end{equation}
and
\begin{equation}\label{eq.V_Nzero}
\langle\omega_N,V_{N}\rangle=\langle H^n(p_N)([N]^*),V_{N}\rangle=\langle[N]^*,H_n(p_N)(V_{N})\rangle=0,
\end{equation}
where by $\langle\omega_X,V_{X}\rangle$ we mean the Kronecker product of $\omega_X$ with any class in $V_X$, for $X=M$ and $N$ in (\ref{eq.V_Mzero}) and (\ref{eq.V_Nzero}) respectively. In particular, 
\begin{equation}\label{eq.deltazero}
\langle\omega_N,\delta\rangle=0.
\end{equation}
By (\ref{eq.fM}) and (\ref{eq.deltazero}) 
$H^n(f)(\omega_N)$ and $\kappa\cdot \omega_M$ coincide on $[M]\otimes 1$:
\begin{equation}\label{eq.M}
\begin{aligned}
\langle H^n(f)(\omega_N),[M]\otimes 1\rangle&=\langle\omega_N,H_n(f)([M]\otimes 1)\rangle\\
&=\langle\omega_N,\kappa\cdot[N]\otimes 1+\delta\rangle\\
&=\langle\omega_N,\kappa\cdot[N]\otimes 1\rangle+\langle\omega_N,\delta\rangle\\
&=\kappa=\langle\kappa\cdot\omega_M,[M]\otimes 1\rangle.
\end{aligned}
\end{equation}
By (\ref{eq.fV_M}), (\ref{eq.V_Mzero}) and (\ref{eq.V_Nzero}), we have 
\begin{equation}\label{eq.V_M}
\langle H^n(f)(\omega_N),V_{M}\rangle=\langle\omega_N,H_n(f)(V_{M})\rangle=0=\langle\kappa\cdot\omega_M,V_{M}\rangle.
\end{equation}
Hence,  by (\ref{Kuenneth}), (\ref{eq.M}) and (\ref{eq.V_M}),
 we have
\[
\langle H^n(f)(\omega_N),z\rangle=\langle\kappa\cdot\omega_M, z\rangle,
\]
for all $z\in H_n(M\times W;\Q)$. By algebraic duality, we obtain (\ref{eq.PDf}).
Note that  (\ref{eq.PDf}) guarantees also that $\kappa\neq0$, because $H^*(f)$ with $\Q$-coefficients is injective, since $\deg(f)=d\neq0$.

The K\"unneth formula in cohomology tells us that
\[
H^m(M\times W;\Q)=\oplus_{i=0}^m (H^{m-i}(M;\Q)\otimes H^{i}(W;\Q)).
\]
 We have
\begin{equation}\label{eq.Z}
H^m(p_Z\circ f)(\omega_Z)=\sum_{i=0}^m \lambda_i(x_{m-i}\times y_i) \in H^m(M\times W;\Q),
\end{equation}
where $x_{m-i}\in H^{m-i}(M;\Q)$, $y_i\in H^i(W;\Q)$ and $\lambda_i\in\Q$.

By (\ref{eq.PDf}), (\ref{eq.Z}),
 the naturality of the cup product and the definition of $d$, we obtain
 \begin{equation*}
 \begin{aligned}
d\cdot\omega_M \times \omega_W&=H^{m+n}(f)(\omega_N \times \omega_Z)\\
&=H^n(f)(\omega_N) \times H^m(f)(\omega_Z)\\
&=\kappa\cdot\omega_M\times \sum_{i=0}^m \lambda_i(x_{m-i}\times y_i)\\
& =\kappa \lambda_m\cdot\omega_M \times \omega_W.
 \end{aligned}
 \end{equation*}
Hence,  $d=\kappa\lambda_m$, and $\lambda_m$ is realised as a mapping degree in $D(W,Z)$ by the map
\[
W\stackrel{\iota_W}\hookrightarrow M\times W\stackrel{f}\longrightarrow N\times Z\stackrel{p_Z}\longrightarrow Z,
\]
Since $d\in D(M\times W,N\times Z)$, $\kappa\in D(M,N)$ and $\lambda_m\in D(W,Z)$, we conclude
\[
D(M\times W, N\times Z)\subseteq D(M,N)\cdot D(W,Z).
\]
\end{proof}

The following fact is also needed  to prove Theorem \ref{productdegrees}.

 \begin{lem}\label{non-prod}\cite[Theorem 1]{Wa1}, \cite[Theorem 1]{KN}
 $K_i$ is dominated by the product of a surface and the circle if and only if $i=0$.
 \end{lem}
 
Now we describe a basis for the third homology group of products of 3-manifolds.

\begin{prop}\label{homologybasis}
Let $Q_1,...,Q_s$ be closed oriented $3$-manifolds and $Q=\prod_{i=1}^{s}Q_i$ be their product. 
Then there is a basis of $H_3(Q;\Q)$, which is represented by the following three classes of closed oriented 3-manifolds in $Q$:
\begin{itemize}
\item[(i)] $Q_1,...,Q_s$.
\item[(ii)] $P_1,...,P_r$, where each $P_i$ is a product of a closed orientable surface and  the circle.
\item[(iii)] Each $3$-manifold which is the $3$-dimensional torus (product of three circles).
\end{itemize}
\end{prop}

\begin{proof}
Let $[Q_i] \in H_3(Q;\mathbb{Q})$ be the integer homology (fundamental) class presented by  $Q_i$ in the $Q$. Denote the first Betti number $b_1(Q_i)$ by $n_i$. 
Suppose that for each $1\le i \le s$
\[
\Sigma_{i,1},\Sigma_{i,2},...,\Sigma_{i, n_i}
\]
is a basis for $H_2(Q_i;\mathbb{Q})$ and
\[
c_{i,1},c_{i,2},...,c_{i,n_i}
\]
is a basis for $H_1(Q_i;\mathbb{Q})$. By the K\"unneth formula in homology we have
\begin{equation*}
\begin{aligned}
H_3(Q_1\times Q_2\times ....\times Q_s; \Q) = & \oplus_{i=1}^s (H_3(Q_i; \Q)\\
&\oplus(\mathop{\oplus}_{\substack{1\le i, j\le s \\ i\ne j}}H_2(Q_i; \Q)\otimes (H_1(Q_j; \Q))\\
&\oplus(\mathop{\oplus}_{\substack{1\le i<j<k \le s
}}H_1(Q_i; \Q)\otimes H_1(Q_j; \Q)\otimes H_1(Q_k; \Q)),
\end{aligned}
\end{equation*}
and the following three homology classes is a basis for $H_3(Q;\mathbb{Q})$:
\begin{itemize}
\item[(i)] $[Q_i], 1\le i \le s$;
\item[(ii)] $\Sigma_{i,i'}\otimes c_{j,j'}, 1\le i, j\le s,\,\,i\ne j\,,\, 1\le i' \le n_i,\,\, 1\le j' \le n_j$;
\item[(iii)] $ c_{i,i'}\otimes c_{j,j'}\otimes c_{k,k'},1\le i<j<k\le s$, 
$1\le i' \le n_i, 1\le j' \le n_{j}, 1\le k' \le n_{k}$.
\end{itemize}

We can always choose $\Sigma_{i,1},\Sigma_{i,2},...,\Sigma_{i,n_i}$ and $c_{i,1},c_{i,2},...,c_{i,n_i}$ to be integer homology classes,
and it is known that in the 3-manifold $Q_i$ any integer homology class $\Sigma_{i,i'}$ of dimension two can be presented by a closed orientable embedded surface $F_{i,i'}$ 
 and each homology class $c_{i,i'}$ of dimension one can be presented by an embedded circle $C_{i,i'}$. Then 
\[
\Sigma_{i,i'}\otimes c_{j,j'}=[F_{i,i'}\times C_{j,j'}],
\]
\[
c_{i,i'}\otimes c_{j,j'}\otimes c_{k,k'}=[C_{i,i'}\times C_{j,j'}\times C_{k,k'}].
\]
This finishes the proof of Proposition \ref{homologybasis}.
\end{proof}

We are now ready to prove Theorem \ref{productdegrees}.

\begin{proof}[Proof of Theorem \ref{productdegrees}]
Let 
\[
q_l>q_{l-1}>q_{l-2}>\cdots>q_2>q_1
\]
be prime numbers such that $q_1>d_l$. 

Following the proof of Theorem \ref{1,k}, let for all $i=1,...,l$
\[
Q_i=(\#_{d_i}K_{q_i})\#K_{d_i}\#K_{d_i^2} \ \ \text{and} \ \ P_i=K_{q_i}\#K_{d_i^2}.
\]
Note that $q_i>d_i$. By (the proof of) Theorem \ref{1,k}, we obtain
\[
D(Q_i,P_i)=\{0,1,d_i\}, \ i=1,...,l.
\]

Let the closed oriented $3l$-manifolds given by the products
 \[
 M=Q_1\times Q_2\times\cdots\times Q_l, \ \ \text{and } \ \ N=P_1\times P_2\times\cdots\times P_l.
\]
By taking products of maps (see Lemma \ref{l:product}), we obtain
\[
\{0,1\}\cup\biggl\{\prod_{j\in S}d_j \ | \ \emptyset\ne S\subseteq \{1,2,...,l\}\biggl\}\subseteq D(M,N).
\]
We thus only need to show that
\[
D(M,N)\subseteq \{0,1\}\cup\biggl\{\prod_{j\in S}d_j \ | \ \emptyset\ne S\subseteq \{1,2,...,l\}\biggl\}.
\]

\noindent{\bf Claim 1:} For each $1\le i \le l-1$, any map 
\[
f_i\colon Q_1\times Q_2 \times\cdots\times Q_i\to P_{i+1}
\]
induces the trivial homomorphism
\[
H_3(f_{i})\colon H_3(Q_1\times Q_2 \times\cdots\times Q_i;\Q)\to H_3(P_{i+1};\Q).
\]

\begin{proof}
Suppose the contrary; then there exists a homology class $h_3\in H_3(Q_1\times Q_2 \times\cdots\times Q_i;\Q)$ and a nonzero integer $d$ such that $H_3(f_{i})(h_3)=d[P_{i+1}]$. We will show that this is impossible.

By Proposition \ref{homologybasis} (and following the notation used in its proof), $h_3$ is a linear combination of the homology classes presented by  $Q_j,\,\, 1\le j \le i$,
$F_{j,j'}\times C_{u,u'}$ and $C_{j,j'}\times C_{u,u'}\times C_{v,v'}$, where $j,j'; u,u'; v, v'$ run over the range as indicated in the proof of 
Proposition \ref{homologybasis}.

Since $P_{i+1}$ is not dominated by a direct product according to Lemma \ref{non-prod}, we have 
\[
H_3(f_{i})([F_{j,j'}\times C_{u,u'}])=0 \ \ \text{and} \ \ H_3(f_{i})([C_{j,j'}\times C_{u,u'}\times C_{v,v'}])=0.
\]
Thus, there exists $1\le r \le i$ such that 
\[
H_3(f_{i})([Q_r])=d'[P_{i+1}]
\]
for some nonzero integer $d'$, that is, there is a $d'$-domination $Q_r\to P_{i+1}$. In particular, Lemma \ref{connectedsumdegrees} implies that
\begin{equation}\label{eq.d'}
0\neq d'\in D(Q_r,P_{i+1})=D(Q_r,K_{q_{i+1}}\# K_{d_{i+1}^2})\subseteq D(Q_r,K_{q_{i+1}}).
\end{equation}
Since $K_{q_{i+1}}$ is aspherical, and so $\pi_2(K_{q_{i+1}})=0$, Lemma \ref{l:degreeconnected} implies that 
\begin{equation*}
\begin{aligned}
D(Q_r,K_{q_{i+1}}) &=D((\#_{d_r}K_{q_r})\#K_{d_r}\#K_{d_r^2}, K_{q_{i+1}})\\
&=\sum^{d_r}D(K_{q_r},K_{q_{i+1}})+D(K_{d_r},K_{q_{i+1}})+D(K_{d_r^2},K_{q_{i+1}}).
\end{aligned}
\end{equation*}
Note that the pairs $(q_{i+1},q_r)$, $(q_{i+1},d_r)$ and $(q_{i+1},d_r^2)$ are all coprime and $q_r,d_r,d_r^2>1$. Hence, by Lemma \ref{l:degreecircle} we obtain
\[
D(K_{q_r},K_{q_{i+1}})=D(K_{d_r},K_{q_{i+1}})=D(K_{d_r^2},K_{q_{i+1}})=\{0\},
\] 
and so $D(Q_r,K_{q_{i+1}})=\{0\}$, which contradicts (\ref{eq.d'}).
\end{proof}

\noindent{\bf Claim 2:} For each $1\le i \le l$,
\begin{equation}\label{eq.claim2}\tag{$\ast$}
D(Q_1\times Q_2 \times\cdots\times Q_i, P_1\times P_2\times\cdots\times P_i)=
\{0,1\}\cup\biggl\{\prod_{j\in S}d_j \ | \ \emptyset\ne S\subseteq \{1,2,...,i\}\biggl\}.
\end{equation}
\begin{proof} We prove the claim by induction. For $i=1$, Theorem \ref{1,k} tells us
\[
D(Q_1,P_1)=\{0,1\}\cup\{d_1\},
\]
therefore (\ref{eq.claim2})  holds.

Suppose that (\ref{eq.claim2}) holds for $i-1$, that is,
\[
D(Q_1\times Q_2\times\cdots\times Q_{i-1},P_1\times P_2\times\cdots\times P_{i-1})=\{0,1\}\cup\biggl\{\prod_{j\in S}d_j \ | \ S\subseteq\{1,2,...,i-1\}\biggl\}.
\]
Note that $P_i$ is not dominated by a direct product (for example because $K_{q_i}$ is not dominated by products; cf. Lemmas \ref{non-prod} and \ref{connectedsumdomination}), and, by Claim 1, any map
\[
f_i\colon Q_1\times Q_2\times\cdots\times Q_{i-1}\rightarrow P_i,
\]
induces the trivial homomorphism
\[
H_3(f_{i})\colon H_3(Q_1\times Q_2 \times\cdots\times Q_{i-1};\mathbb{Q})\to H_3(P_{i};\mathbb{Q}).
\]
Thus, $P_i$ satisfies conditions (i) and (ii) of Theorem \ref{t:product} (for $W=Q_1\times\cdots\times Q_{i-1}$), and therefore Theorem \ref{t:product} implies (for $Z=P_1\times\cdots\times P_{i-1}$)
\[
D(Q_1\times Q_2 \times\cdots\times Q_i,P_1\times P_2\times\cdots\times P_i)
=D(Q_1\times Q_2 \times\cdots\times Q_{i-1},P_1\times P_2\times\cdots\times P_{i-1})\cdot 
D(Q_i,P_i).
\]
By the induction hypothesis and Theorem \ref{1,k}, it follows that
\begin{equation*}
\begin{aligned}
& D(Q_1\times Q_2\times\cdots\times Q_i,P_1\times P_2\times\cdots\times P_i)\\
&=\biggl(\{0,1\}\cup\biggl\{\prod_{j\in S}d_j \ | \ S\subseteq \{1,2,...,i-1\biggl\}\biggl)\cdot \{0,1,d_i\}\\
&=\{0,1\}\cup\biggl\{\prod_{j\in S}d_j \ | \ S\subseteq \{1,2,...,i\}\biggl\}.
\end{aligned}
\end{equation*}
Hence (\ref{eq.claim2}) holds for $i$. This finishes the proof of Claim 2.
\end{proof}
Theorem \ref{productdegrees} follows as a special case of Claim 2 for $i=l$.
\end{proof}

\bibliographystyle{amsalpha}

\end{document}